\documentclass[reqno,11pt]{amsart}
\usepackage[margin=1in]{geometry}

\usepackage{amsthm, amsmath, amssymb, bm}
\usepackage{thmtools}
\usepackage{thm-restate}
\usepackage{tikz}
\usepackage{mathtools}

\usepackage{microtype}

\usepackage[utf8]{inputenc}
\usepackage[T1]{fontenc}

\usepackage[textsize=small,backgroundcolor=orange!20]{todonotes}

\usepackage[hidelinks]{hyperref}

\usepackage[noabbrev,capitalize]{cleveref}
\crefformat{equation}{(#2#1#3)}
\crefmultiformat{equation}{(#2#1#3)}{ and~(#2#1#3)}{, (#2#1#3)}{ and~(#2#1#3)}
\numberwithin{equation}{section}

\usepackage{url}

\usepackage[color,final]{showkeys} 

\colorlet{refkey}{orange!20}
\colorlet{labelkey}{blue!60}

\newtheorem{theorem}{Theorem}[section]

\newtheorem{lemma}[theorem]{Lemma}

\newtheorem{corollary}[theorem]{Corollary}

\theoremstyle{definition}

\newtheorem{construction}[theorem]{Construction}

\newtheorem{question}[theorem]{Question}

\theoremstyle{remark}
\newtheorem{remark}[theorem]{Remark}

\newcommand{\abs}[1]{\left\lvert#1\right\rvert}

\newcommand{\norm}[1]{\left\lVert#1\right\rVert}

\newcommand{\ang}[1]{\left\langle #1 \right\rangle}
\newcommand{\floor}[1]{\left\lfloor #1 \right\rfloor}
\newcommand{\ceil}[1]{\left\lceil #1 \right\rceil}
\newcommand{\paren}[1]{\left( #1 \right)}
\newcommand{\sqb}[1]{\left[ #1 \right]}

\newcommand{\Cay}{\operatorname{Cay}}

\newcommand{\EE}{\mathbb{E}}
\newcommand{\FF}{\mathbb{F}}
\newcommand{\CC}{\mathbb{C}}

\newcommand{\RR}{\mathbb{R}}

\newcommand{\ZZ}{\mathbb{Z}}

\newcommand{\edit}[1]{{\ttfamily\upshape\small\color{red}[#1]}}

\DeclareMathOperator{\Ind}{Ind}
\newcommand{\triv}{\mathrm{triv}}

\DeclareMathOperator{\SL}{SL}
\DeclareMathOperator{\PSL}{PSL}

\title{Graphs with high second eigenvalue multiplicity}

\author[Haiman]{Milan Haiman}
\author[Schildkraut]{Carl Schildkraut}
\author[Zhang]{Shengtong Zhang}
\author[Zhao]{Yufei Zhao}

\thanks{
Zhao was supported by NSF award DMS-1764176, NSF CAREER award DMS-2044606, a Sloan Research
Fellowship, and the MIT Solomon Buchsbaum Fund.
}

\address{Department of Mathematics, Massachusetts Institute of Technology, Cambridge, MA 02139, USA}
\email{\{mhaiman,carlsc,stzh1555,yufeiz\}@mit.edu}

\begin{document}

\begin{abstract}
	Jiang, Tidor, Yao, Zhang, and Zhao recently showed that connected bounded degree graphs have sublinear second eigenvalue multiplicity (always referring to the adjacency matrix).
	This result was a key step in the solution to the problem of equiangular lines with fixed angles.
	It led to the natural question: what is the maximum second eigenvalue multiplicity of a connected bounded degree $n$-vertex graph?
	The best known upper bound is $O(n/\log\log n)$.
	The previously known best known lower bound is on the order of $n^{1/3}$ (for infinitely many $n$), coming from Cayley graphs on $\PSL(2,q)$.
	
	Here we give a construction showing a lower bound of $\sqrt{n/\log_2 n}$. We also construct Cayley graphs with second eigenvalue multiplicity at least $n^{2/5}-1$. 
	
	Earlier techniques show that there are at most $O(n/\log\log n)$ eigenvalues (counting multiplicities) within $O(1/\log n)$ of the second eigenvalue. 
	We give a construction showing this upper bound on approximate second eigenvalue multiplicity is tight up to a constant factor.
	This demonstrates a barrier to earlier techniques for upper bounding eigenvalue multiplicities.
\end{abstract}

\maketitle

\section{Introduction}

\subsection{Second eigenvalue multiplicity}

Jiang, Tidor, Yao, Zhang, and Zhao \cite{JTYZZ1} recently determined the maximum number of lines in high dimensions pairwise forming a fixed given angle. 
As a key ingredient in their solution, they proved a novel result in spectral graph theory: connected bounded degree graphs have sublinear second eigenvalue multiplicity. 

Throughout this paper, we always refer to the \emph{adjacency matrix} of a graph $G$ when we talk about its eigenvalues. We write $\lambda_1(G) \ge \lambda_2(G) \ge \cdots \ge \lambda_{\abs{V(G)}}(G)$ for the eigenvalues of adjacency matrix of $G$ listed in descending order. We call $\lambda_2(G)$ its \emph{second eigenvalue} and the number of times it appears in the list as the \emph{second eigenvalue multiplicity}.

\begin{theorem}[{\cite{JTYZZ1}}]
\label{thm:JYTZZ}
For every $\Delta$, there is a constant $C = C(\Delta)$ so that every connected $n$-vertex graph with maximum degree at most $\Delta$ has second eigenvalue multiplicity at most $C n /\log\log n$.
\end{theorem}

In spectral graph theory, the spectral gap has been intensely studied due to its connections to graph expansion (see, e.g., the surveys on expander graphs~\cite{HLW06,Lub12}). 
In contrast, much less is known about second eigenvalue multiplicities.
\cref{thm:JYTZZ} was the first such result about general classes of graphs.

In light of \cref{thm:JYTZZ}, the following question was raised in \cite{JTYZZ1}.

\begin{question} \label{q:main}
    For fixed $\Delta$ and large $n$, what is the maximum second eigenvalue multiplicity among connected $n$-vertex graphs with maximum degree at most $\Delta$?
	What about among regular graphs? Cayley graphs?
\end{question}

\begin{remark}
(a) Among expander graphs, i.e., with some constant $c>0$ so that $\abs{N(A)}\ge (1+c)\abs{A}$ for all vertex sets with $\abs{A}\le n/2$, the proof of \cref{thm:JYTZZ} gives a slightly better upper bound of $O_{\Delta, c}(n/\log n)$.

(b) Lee and Makarychev~\cite{LM08} showed that if $G$ is a connected Cayley graph with doubling constant $K = \max_{R > 0} \abs{B(2R)}/\abs{B(R)}$, then its second eigenvalue multiplicity is at most $K^{O(\log K)}$.

(c) Among regular graphs, the upper bound in \cref{thm:JYTZZ} was improved to $n/(\log n)^{1/5-o(1)}$ by
McKenzie, Rasmussen, and Srivastava~\cite{MRS21}. 

(d) As noted in \cite{JTYZZ1}, \cref{thm:JYTZZ} also applies to the $j$-th eigenvalue of $G$ for any fixed $j$. While the second eigenvalue multiplicity result is sufficient for equiangular lines, the $j$-th eigenvalue multiplicity result is needed for a generalization to spherical two-distance sets~\cite{JTYZZ2}. 
\end{remark}

The previously best known example (i.e., lower bound) for all three questions is any connected 4-regular Cayley graph on $\PSL(2,q)$. It has $n = q(q^2-1)/2$ vertices. Since all non-trivial representations of $\PSL(2,q)$ have dimension at least $(q-1)/2$, all eigenvalues of the graph except for the top one must appear with multiplicity at least $(q-1)/2 = \Theta(n^{1/3})$. This gives a lower bound of $cn^{1/3}$ (for infinitely many $n$). One of the goals of this paper is to improve this lower bound using new constructions.

\begin{theorem} \label{thm:cayley-intro}
There exist infinitely many connected $n$-vertex $18$-regular Cayley graphs with second eigenvalue multiplicity at least $n^{2/5}-1$.
\end{theorem}

\begin{theorem} \label{thm:bounded-intro}
There exist infinitely many connected $n$-vertex graphs with maximum degree $4$ and second eigenvalue multiplicity at least $\sqrt{n /\log_2 n}$.
\end{theorem}

Our constructions use large irreducible group representations to guarantee high multiplicities. 
However, not all eigenvalues in our constructed graphs have high multiplicity, unlike the more straightforward $\PSL(2,q)$ Cayley graphs earlier.
We manipulate our graphs to ensure that the desired high multiplicity eigenvalues are indeed the second largest eigenvalues.
It would be very interesting to obtain constructions with second eigenvalue multiplicity above $\sqrt{n}$, which appears to be a barrier for group representation based methods, since every irreducible representation of an order $n$ group has dimension at most $\sqrt{n}$.

There is still a very large gap between the best known upper and lower bounds for \cref{q:main}. 
In particular, it would be interesting to determine whether there is an upper bound of the form $n^{1-c}$ for some constant $c > 0$.
In addition, an improved upper bound in \cref{thm:JYTZZ} would imply quantitative improvements on the equiangular lines result in \cite{JTYZZ1}.
Indeed, the main theorem in \cite{JTYZZ1} for equiangular lines applies to sufficiently large dimensions, and we would be able to lower the ``sufficiently large'' requirement if we could improve \cref{thm:JYTZZ}.

\subsection{Approximate second eigenvalue multiplicity}

It was observed in \cite{MRS21} that existing proof techniques also give an upper bound on approximate second eigenvalue multiplicities. 
Given $I \subset \RR$ (usually an interval), we write $m_G(I)$ for the number of eigenvalues of $G$ (counting multiplicities) that lie in $I$. 
In particular, we write
\[
m_G[a,b] := \#\{ j : \lambda_j(G) \in [a,b]\}.
\]
The proof in \cite{JTYZZ1} can be easily altered to show the following (see \cref{app:proof} for a brief explanation).

\begin{theorem}[{\cite{JTYZZ1}}]\label{thm:JYTZZ-approx}
For every $\Delta$ and any $K > 0$, there is a constant $C = C(\Delta, K)$ so that every connected $n$-vertex graph $G$ with maximum degree at most $\Delta$ satisfies
\[
m_G \sqb{\paren{1 - \frac{K}{\log n}} \lambda_2(G), \lambda_2(G)} \le \frac{C n}{\log\log n}.
\]
\end{theorem}

We give a construction showing that the above bound is in some sense best possible up to a constant factor. In comparison with \cref{thm:JYTZZ-approx}, this construction demonstrates a barrier for using the trace method techniques of \cite{JTYZZ1} to further improve \cref{thm:JYTZZ}.

We define
\[
\textit{the spectral gap of $G$} = \lambda_1(G) - \lambda_2(G).
\]
(Unlike in some contexts, here we do not normalize by the degree. Also, we always consider the adjacency matrix and not the graph Laplacian.)

\begin{theorem} \label{thm:approx-intro}
For any $C, \delta > 0$ there exists some $c > 0$ such that there exist infinitely many connected $n$-vertex graphs $G$ with maximum degree $6$, spectral gap at least $C /\log n$, and
\[
m_G\sqb{\paren{1 - \frac{\delta}{\log n}} \lambda_2(G), \lambda_2(G)}
\ge \frac{c n}{\log\log n}.
\]
\end{theorem}

Analogous results were shown by McKenzie, Rasmussen and Srivastava \cite{MRS21} for regular graphs. 
They showed that the approximate second eigenvalue multiplicity of a bounded degree regular graph is at most $n/(\log n)^{1/5-o(1)}$, and they also gave a construction with $\Omega(n/(\log n)^{3/2})$ approximate second eigenvalues.

\subsection*{Organization}
We give the Cayley graph construction for \cref{thm:cayley-intro} in \cref{sec:cayley}, followed by the bounded degree graph construction for \cref{thm:bounded-intro} in \cref{sec:bounded}.
We give the construction and proof for \cref{thm:approx-intro} in \cref{sec:approx1}.

\section{Cayley graph construction} \label{sec:cayley}

In this section we give our construction proving \cref{thm:cayley-intro}, that there are infinitely many connected $n$-vertex 18-regular Cayley graphs with second eigenvalue multiplicity at least $n^{2/5}-1$. 
Here the \emph{Cayley graph} $\Cay(\Gamma,S)$ on a group $\Gamma$ generated by $S = S^{-1}$ has vertex set $\Gamma$ and edges $(g, gs)$ for all $g \in \Gamma$ and $s \in S$. 
By the Lubotzky--Phillips--Sarnak construction of Ramanujan graphs, there are infinitely many prime powers $q$ for which there exists a $8$-regular Cayley graph on $\operatorname{PSL}(2,q)$ with spectral gap at least $8-2\sqrt 7>2$. 
We will lift this graph to a Cayley graph $\Cay(\operatorname{SL}(2,q),S)$ with $|S|=16$ and spectral gap at least $4$. (The precise lifting procedure is described in the proof of \cref{thm:cayley-intro} later in this section.) 
Now, let $\Gamma=\operatorname{SL}(2,q) \ltimes \FF_q^2$, defined via the standard action of $\SL(2,q)$ on $\FF_q^2$, and let $t$ be any element of $\Gamma$ not in $\operatorname{SL}(2,q)$, with $t\neq t^{-1}$. 
The graph we construct will be $\Cay(\Gamma,S\cup\{t,t^{-1}\})$. 
This graph has $q^3(q^2 - 1)$ vertices. 
We will show that its second eigenvalue multiplicity is at least $q^2 - 1$.

\medskip 

Given the adjacency matrix $A$ of a Cayley graph $\Cay(\Gamma, S)$, and a representation $\rho: \Gamma \to \text{GL}_n(\mathbb{C})$ of $\Gamma$, we write
\[
A_\rho = \sum_{s \in S} \rho(s),
\]
which can be viewed as a $\dim \rho \times \dim \rho$ matrix. Then $A = A_{\mathrm{reg}}$, where reg is the regular representation. 
Since the regular representation of $\Gamma$ contains $\dim\rho$ copies of each irreducible representation $\rho$ of $\Gamma$, the matrix $A$ can be block-diagonalized into $\dim\rho$ copies of $A_\rho$ for each irreducible representation $\rho$.

Here is our general recipe for constructing Cayley graphs with large second eigenvalue multiplicity. Define $\Pi\backslash\Gamma/\Pi = \{\Pi g \Pi : g \in \Gamma\}$ as the set of double cosets of $\Pi$.

\begin{theorem} \label{thm:cayley-general}
Let $\Gamma$ be a finite group.
Let $S = S^{-1}$ be a symmetric subset of $\Gamma$.
Let $\Pi$ be the subgroup of $\Gamma$ generated by $S$.
Suppose $\Cay(\Pi, S)$ has spectral gap at least $4$.
Suppose that $\abs{\Pi\backslash\Gamma/\Pi} = 2$.
Let $t$ be an element of $\Gamma$ not in $\Pi$, with $t\neq t^{-1}$.
Then $\Cay(\Gamma, S\cup \{t,t^{-1}\})$ has second eigenvalue multiplicity at least
$\abs{\Gamma}/\abs{\Pi}-1$.
\end{theorem}

Write $\triv$ for the trivial 1-dimensional representation of a group. Recall that the induction $\Ind_\Pi^\Gamma(\triv)$ of $\triv$ on $\Pi$ to $\Gamma$ is also the permutation representation of $\Gamma$ on left $\Pi$-cosets.

\begin{lemma}
	\label{lem:double-coset-irrep}
	Let $\Gamma$ be a finite group and $\Pi$ a subgroup.
	Suppose $\abs{\Pi\backslash\Gamma/\Pi}=2$.
	Then $\Ind_\Pi^\Gamma(\triv)$ is the direct sum of the trivial representation and an irreducible representation of $\Gamma$ of dimension $\abs{\Gamma}/\abs{\Pi}-1$.
\end{lemma}

The lemma is an immediate consequence of the following lemma, which implies that $\Ind_\Pi^\Gamma(\triv)$ has exactly two irreducible components.

\begin{lemma}\label{lem:char} 
	Let $\Gamma$ be a finite group and $\Pi$ a subgroup.
	The character $\chi$ of $\Ind_\Pi^\Gamma(\triv)$ satisfies
	\[
	\ang{\chi, \chi} = \abs{\Pi\backslash \Gamma / \Pi}.
	\]
\end{lemma}

This is a direct consequence of Mackey theory \cite[Sections 7.3--7.4]{Ser77}. In addition, we offer a direct proof.

\begin{proof} 
Since $\Ind_\Pi^\Gamma(\triv)$ is a permutation representation, for each $g \in \Gamma$, $\abs{\chi(g)}^2$ is the number of fixed points of $g$ on $\Gamma/\Pi \times \Gamma/\Pi$ (where $g$ sends $(a\Pi, b\Pi)$ to $(ga\Pi, gb\Pi)$). By Burnside's lemma, $\ang{\chi, \chi} = \EE_{g \in \Gamma} \abs{\chi(g)}^2$ is the number of orbits in $\Gamma/\Pi \times \Gamma/\Pi$ under the action of $\Gamma$. 
Note that the orbit of $(a \Pi, b \Pi)$ contains $(b^{-1} a \Pi, \Pi)$.
Furthermore, $(a \Pi, \Pi)$ and $(b\Pi, \Pi)$ lie in the same orbit if and only if $a \in \Pi b$. It follows that the number of orbits equals to $\abs{\Pi\backslash \Gamma / \Pi}$.
\end{proof}

Recall Weyl's inequality on perturbation of matrix eigenvalues. Denote the eigenvalues of an $n\times n$ real symmetric  matrix $A$ by $\lambda_1(A) \ge \lambda_2(A) \ge \cdots \ge \lambda_n(A)$.

\begin{theorem}[Weyl's inequality]
	If $A$ and $B$ are real symmetric $n\times n$ matrices, then
	\[
	\abs{\lambda_i(A+B) - \lambda_i(A)} \le \norm{B} \qquad \text{for each }i = 1, \dots, n.
	\]
\end{theorem}

\begin{proof}[Proof of \cref{thm:cayley-general}]
	Let $A$ and $B$ be the adjacency matrices of $\Cay(\Gamma, S)$ and $\Cay(\Gamma, \{t,t^{-1}\})$.
	
	We first claim that, given an irreducible representation $\rho$ of $\Gamma$, the top eigenvalue of $A_\rho$ is equal to that of $A$ (namely $\abs{S}$) if and only if $\rho$ is contained in $\Ind_\Pi^\Gamma(\triv)$.
	
	We write $\operatorname{Res}_\Pi^\Gamma(\rho)$ for the restricted representation, i.e., the restriction of $\rho \colon \Gamma \to \operatorname{GL}_{\dim \rho}(\CC)$ to $\Pi$.
	Since $S\subset\Pi$, $A_\rho$ can be decomposed in accordance with the decomposition $\operatorname{Res}_\Pi^\Gamma(\rho)$ into irreducible representations of $\Pi$. As the adjacency matrix of $\Cay(\Pi,S)$ can be decomposed in a similar manner, and only the trivial representation gives an eigenvalue of $\abs{S}$, the only $\rho$ for which $A_\rho$ has an eigenvalue of $\abs{S}$ will be those for which $\triv_\Pi$ is contained in $\operatorname{Res}_\Pi^\Gamma(\rho)$. The claim then follows from Frobenius reciprocity. 
	
	Now, the top eigenvalue of $A+B$ is equal to $\abs{S \cup \{t, t^{-1}\}}$, coming from the trivial 1-dimensional representation of $\Gamma$.
	By \cref{lem:double-coset-irrep}, there is exactly one other irreducible representation $\varphi$ of $\Ind_\Pi^\Gamma(\triv)$.
	The top eigenvalue of $A_\varphi$ is $\abs{S}$.
	So by Weyl's inequality, since $\norm{B_\varphi} \le 2$, the top eigenvalue $\mu$ of $A_\varphi + B_\varphi$ is at least $\abs{S} - 2$.
	Applying Weyl's inequality again and using the spectral gap $\lambda_1(\Cay(\Pi,S)) - \lambda_2(\Cay(\Pi,S)) \ge 4$, all eigenvalues of $A+B$ besides $\abs{S \cup \{t, t^{-1}\}}$ and $\mu$ are at most
	$$\lambda_2(\Cay(\Pi,S)) + 2 \le \lambda_1(\Cay(\Pi,S)) - 2 = \abs{S} - 2.$$
	It follows that $\mu$ is indeed the second largest eigenvalue of $A+B$, and it has multiplicity at least $\dim \varphi \ge \abs{\Gamma}/\abs{\Pi} - 1$.
\end{proof}

Now we apply the above construction to groups of the form $\Gamma = \Pi \ltimes V$ where $V$ is some group that admits an action by $\Pi$. Then $\abs{\Pi \backslash \Gamma / \Pi}$ is the number of $\Pi$-orbits in $V$. 
We want to select $S \subset \Pi$ so that $\Cay(\Pi, S)$ has spectral gap at least $4$, and $V$ has exactly two $\Pi$-orbits.

Recall the construction of Ramanujan graphs by Lubotzky, Phillips, and Sarnak~\cite{LPS88} and Margulis~\cite{Mar89}, with later extension by Morgenstern~\cite{Mor94}.

\begin{theorem}[Ramanujan graph] \label{thm:ram}
	For every prime power $p$, there exist infinitely many prime powers $q$ for which there exists a $(p+1)$-regular connected Cayley graph on $\PSL(2,q)$ all of whose eigenvalues, other than the largest, are at most $2\sqrt{p}$ in absolute value.
\end{theorem}

Now we are ready to construct a family of bounded degree $n$-vertex Cayley graphs with second eigenvalue multiplicity at least $n^{2/5}-1$.

\begin{proof}[Proof of \cref{thm:cayley-intro}]
By \cref{thm:ram} applied with $p = 7$, for infinitely many $q$, there exists a $8$-regular connected Cayley graph $H_0=\Cay(\PSL(2,q),S_0)$ with spectral gap at least $8-2\sqrt{7}>2$. Let $S$ be the preimage of $S_0$ under the quotient map from $\Pi=\SL(2,q)$ to $\PSL(2,q)$ so that $|S|=16$.
We claim the spectral gap of $H=\Cay(\SL(2,q),S)$ is at least $4$. $H$ can be constructed from $H_0$ by replacing each vertex $v$ with a pair $(x_v,y_v)$ and adding edges $x_vx_w$, $x_vy_w$, $y_vx_w$, and $y_vy_w$ for each edge $vw$ of $E_0$. As a result, any eigenvector $u$ of $A_{H_0}$ with eigenvalue $\lambda$ gives an eigenvector $u'$ of $A_H$ with $u'_{x_v}=u'_{y_v}=u_v$ and eigenvalue $2\lambda$. In addition, the vector that is $1$ at $x_v$ and $-1$ at $y_v$, and $0$ at all other vertices of $H$, has eigenvalue $0$. So, the eigenvalues of $A_H$ are exactly twice those of $A_{H_0}$ (with appropriate multiplicity), along with $\abs\Pi/2$ copies of $0$. Therefore the spectral gap of $H$ is double that of $H_0$, and is thus at least $4$. (This can also be seen by writing the $H$ as the tensor product of $H_0$ and \tikz[baseline = -0.7ex, scale = .8, every loop/.style={}] {
    \node [name=1,fill=black,circle,minimum size=2.5pt,inner sep=0pt,outer sep=0pt] {} edge [in=-130,out=130,loop] (1);
    \node [name=2,fill=black,circle,minimum size=2.5pt,inner sep=0pt,outer sep=0pt] (2) [right = 0.25cm of 1] {} edge [in=50,out=-50,loop] ();
    \draw[-](1) -- (2);
    }
\!\!.)

Now, using the standard action of $\Pi$ on the additive group $\FF_q^2$, 
set $\Gamma = \Pi \ltimes \FF_q^2$.
Since $\Pi$ is transitive on $\FF_q^2 \setminus\{(0,0)\}$, we have 
$\abs{\Pi \backslash \Gamma / \Pi} = 2$.
Then applying \cref{thm:cayley-general} gives a $18$-regular Cayley graph on $\Gamma$ whose second eigenvalue multiplicity is at least $q^2 - 1 \ge \abs{\Gamma}^{2/5} - 1$.
\end{proof}

\begin{remark} In the proof of \cref{thm:cayley-intro}, we may replace the pair $(\Pi,S)$ with any subgroup $\Pi$ of $\mathrm{GL}(2,q)$ that acts transitively on $\mathbb F_q^2$ and where $\Cay(\Pi,S)$ has spectral gap at least $4$. This gives a graph with the desired property with second eigenvalue multiplicity at least $q^2-1$ on $q^2\abs\Pi$ vertices. In particular, if $\Pi$ can be chosen with $\abs{\Pi}=O(q^{2+c})$ for some $0\leq c<1$, this will improve the second eigenvalue multiplicity bound to $\Omega(n^{2/(4+c)})$.

Furthermore, if $\Cay(\Pi,S)$ has spectral gap at least $\epsilon$ for some constant $\epsilon$ bounded away from $0$ as $q$ grows, we may apply an augmenting strategy to increase the spectral gap to $4$. Select $N=\lceil 4/\epsilon\rceil$, replace $\Pi$ with $\Pi\times \ZZ/N\ZZ$ where the $\ZZ/N\ZZ$ factor acts trivially on $\mathbb F_q^2$, and replace $S$ with $S\times \ZZ/N\ZZ$. Since the eigenvalues of $\Cay(\Pi\times \ZZ/N\ZZ,S\times \ZZ/N\ZZ)$ are exactly those of $\Cay(\Pi,S)$ multiplied by $N$, along with many copies of $0$, this gives a Cayley graph with spectral gap $N\epsilon>4$, at the cost of scaling the degree and the number of vertices by a constant $\Theta(1/\epsilon)$. So, we only require that $\Pi$ be \emph{expanding}, i.e. that it possess Cayley graphs of bounded degree with spectral gap bounded away from $0$.
\end{remark}

\section{Bounded degree graph construction} \label{sec:bounded}

In this section we give our construction proving \cref{thm:bounded-intro} that there are infinitely many connected $n$-vertex graphs with maximum degree 4 and second eigenvalue multiplicity at least $\sqrt{n/\log_2 n}$.
Take a prime power $q$ and set $\Gamma=\FF_q^\times\ltimes\FF_q$, which is isomorphic to the affine group of $\FF_q$. 
Let $s$ be a generator of $\FF_q^\times$.
Let $t$ be any element of $\Gamma\setminus\FF_q^\times$ with $t\neq t^{-1}$. 
Take $\Cay(\Gamma,\{s,s^{-1},t,t^{-1}\})$ and replace every edge generated by $t$ or $t^{-1}$ by a path of $m = \ceil{2 \log_2(q-1)} - 2$ edges. 
This graph has $q(q-1)m$ vertices.
We will show that it has second eigenvalue multiplicity at least $q-1$.

\medskip 

Let us state the construction more generally.

\begin{construction}\label{const:bounded-construction} Let $\Gamma$ be a finite group. Let $S = S^{-1}$ be a symmetric subset of $\Gamma$ with $\abs{S} \ge 2$, and let $\Pi$ be the subgroup of $\Gamma$ generated by $S$.
Suppose that $\abs{\Pi\backslash\Gamma/\Pi} = 2$.
Let $t$ be an element of $\Gamma$ outside $\Pi$ with $t \ne t^{-1}$.
Define $G=G(\Gamma, S, t, m)$ to be the graph obtained from $\Cay(\Gamma, S \cup \{t, t^{-1}\})$ by replacing each edge generated by $t$ or $t^{-1}$ by a path of $m$ edges.
\end{construction}

\begin{theorem}\label{thm:bounded-construction} Let $\Gamma$, $S$, $t$, and $m$ be as in \cref{const:bounded-construction}. Suppose $\Cay(\Pi, S)$ has spectral gap $\kappa$. Let $m\geq 4$ be a positive integer satisfying $\abs{S}^{m-1} \ge 4/\kappa$.

Then $G$ is a connected graph on $\abs{\Gamma}m$ vertices with maximum degree $\abs{S} + 2$ and second eigenvalue multiplicity at least
$\abs{\Gamma}/\abs{\Pi}-1$.
\end{theorem}

Recall the \emph{Chebyshev polynomials of the first kind} $T_m$, defined by the recurrence relation
\begin{equation}\label{eq:T-recur}
T_0(x)=1, \quad
T_1(x)=x, \quad  \text{and} \quad
T_{m+1}(x)=2xT_m(x)-T_{m-1}(x) \text{ for all } m \ge 1.
\end{equation}
They satisfy
\begin{equation}
	\label{eq:T-alpha}
T_m\paren{\frac{\alpha+\alpha^{-1}}2}
=\frac{\alpha^m+\alpha^{-m}}{2}.
\end{equation}
The \emph{Chebyshev polynomials of the second kind} $U_m$ are defined by the recurrence relation
\begin{equation}
\label{eq:U-recur}	
U_0(x)=1, \quad
U_1(x)=2x, \quad
U_{m+1}(x)=2xU_m(x)-U_{m-1}(x) \text{ for all } m \ge 1.
\end{equation}
They satisfy
\begin{equation}
	\label{eq:U-alpha}
U_m\paren{\frac{\alpha+1/\alpha}2}
= \frac{\alpha^{m+1}-\alpha^{-m-1}}{\alpha- \alpha^{-1}}.
\end{equation}
We also have
\begin{equation}\label{eq:T-U}
T_m(x)=xU_{m-1}(x)-U_{m-2}(x).
\end{equation}
In addition, we will need the following lemma.
\begin{lemma}\label{lem:cheb-facts}
\begin{enumerate}
    \item[(a)] The function
    \[z\mapsto \frac{T_m(z)-1}{U_{m-1}(z)}\]
    is increasing for $z\geq 1$.
    \item[(b)] For $m\geq 4$ and $z\geq 1$,
    \[U_{m-1}(z)\geq \left(\frac{2T_m(z)+2}{U_{m-1}(z)}\right)^{m-1}.\]
\end{enumerate}
\end{lemma}
\begin{proof} We parameterize $z=(\alpha+\alpha^{-1})/2$ with $\alpha \ge 1$. By \cref{eq:T-alpha,eq:U-alpha}, 
\begin{align*}
\frac{T_m\paren{\frac{\alpha + \alpha^{-1}}{2}}\pm 1}{U_{m-1}\paren{\frac{\alpha + \alpha^{-1}}{2}}}
&= \frac{(\alpha - \alpha^{-1})(\alpha^m + \alpha^{-m} \pm 2)}{2(\alpha^m - \alpha^{-m})}
= \frac{(\alpha - \alpha^{-1})
(\alpha^{m/2} \pm \alpha^{-m/2})^2}{2(\alpha^m - \alpha^{-m})}
\\
&= \frac{(\alpha - \alpha^{-1})(\alpha^{m/2} \pm \alpha^{-m/2})}{2(\alpha^{m/2} \mp \alpha^{-m/2})}
= \frac{\alpha-\alpha^{-1}}{2} \cdot \frac{1\pm \alpha^{-m}}{1\mp \alpha^{-m}}.
\end{align*}
\begin{enumerate}
\item[(a)] In the case of $T_m(z)-1$, the function is increasing in $\alpha \ge 1$ (and hence increasing in $z = (\alpha + \alpha^{-1})/2 \ge 1$) since it is a product of three increasing functions.
\item[(b)] Using $m\geq 4$,
\[
\frac{2T_m(z)+2}{U_{m-1}(z)}
=(\alpha-\alpha^{-1})\frac{1+\alpha^{-m}}{1-\alpha^{-m}}
\leq (\alpha-\alpha^{-1})\frac{1+\alpha^{-4}}{1-\alpha^{-4}}
=\frac{\alpha^4+1}{\alpha^3+\alpha}\leq \alpha.
\]
The result follows from using \cref{eq:U-alpha} to derive
$$\alpha^{m-1}\leq \frac{\alpha^m-\alpha^{-m}}{\alpha-\alpha^{-1}}=U_{m-1}(z).\qedhere$$
\end{enumerate}

\end{proof}

The following identity tell us how to analyze the eigenvalues of a graph after subdividing the edges of a regular subgraph.

\begin{lemma}\label{lem:pathlen} Let $H_1$ and $H_2$ be graphs on the same vertex set $V$.
Suppose $H_2$ is $d$-regular.
Let $G$ be obtained by first overlaying $H_1$ and $H_2$ (initially viewing the union as a multigraph) and then replacing each edge of $H_2$ with a path of $m$ edges. 
Then
\[
\det(A_G- 2x I)=\pm U_{m-1}(x)^{e(H_2)}\det\left(A_{H_1}-\left(2x-\frac{dU_{m-2}(x)}{U_{m-1}(x)}\right)I+\frac{1}{U_{m-1}(x)}A_{H_2}\right).
\]
\end{lemma}

\begin{proof}

For each edge in $H_2$, we will perform a sequence of row and column operations on $A_G-2xI$ that does not change the determinant.

Consider an edge $vw\in E(H_2)$, and let $v=x_0,x_1,\dots,x_{m-1},x_m=w$ be the $m$-edge path connecting $v$ and $w$ in $G$. Let $A_{uv}$ be the $(m+1)$ by $(m+1)$ submatrix of $A_G-2xI$ corresponding to the vertices $(x_0,\dots,x_m)$. Let $r_i$ be the row vector in $A_G - 2x I$ corresponding to vertex $x_i$ of $G$ and $r_{ij}$ be the entry of $r_i$ corresponding to vertex $x_j$. 

Before performing the operations for this edge, $A_{uv}$ is $$\begin{bmatrix}
*       &1      &0      &       &\cdots &       &0      \\
1       &-2x    &1      &0      &       &       &       \\
0       &1      &-2x    &1      &0      &       &\vdots \\
        &0      &1      &\ddots &\ddots &\ddots &       \\
\vdots  &       &0      &\ddots &-2x    &1      &0      \\
        &       &       &\ddots &1      &-2x    &1      \\
0       &       &\cdots &       &0      &1      &*      \\
\end{bmatrix}.$$
Our goal is to zero the outer rows and columns of $A_{uv},$ excluding the corners.

To zero $r_{01}$ while keeping the zero entries zero, we add $$\frac{1}{U_{m-1}(x)}\sum_{i=1}^{m-1}U_{m-1-i}(x)r_i$$ to $r_0$. We calculate using \eqref{eq:U-recur} that 
$$\frac{1}{U_{m-1}(x)}\sum_{i=1}^{m-1}U_{m-1-i}(x)r_{ij}=\frac{1}{U_{m-1}(x)}\begin{cases}U_{m-2}(x)&\text{if }j=0\\-U_{m-1}(x)&\text{if }j=1\\1&\text{if }j=m\\0&\text{otherwise,}\end{cases}$$
so this operation zeros $r_{01}$ as desired.

Similarly, to zero $r_{m,m-1}$, we add $$\frac{1}{U_{m-1}(x)}\sum_{i=1}^{m-1}U_{i-1}(x)r_i$$ to $r_m$.


So, $A_{uv}$ is now
$$\begin{bmatrix}
* +\frac{U_{m-2}(x)}{U_{m-1}(x)}      &0      &0      &       &\cdots &0       &\frac{1}{U_{m-1}(x)}      \\
1       &-2x    &1      &0      &       &       &0       \\
0       &1      &-2x    &1      &0      &       &\vdots \\
        &0      &1      &\ddots &\ddots &\ddots &       \\
\vdots  &       &0      &\ddots &-2x    &1      &0      \\
0       &       &       &\ddots &1      &-2x    &1      \\
\frac{1}{U_{m-1}(x)}       &0      &\cdots &       &0      &0      &* +\frac{U_{m-2}(x)}{U_{m-1}(x)}      \\
\end{bmatrix}.$$

Next, we perform similar column operations to zero $r_{10}$ and $r_{m-1,m}$ Since we already have $r_{01}=r_{m,m-1}=0$, these column operations do not change the corner entries of $A_{uv}$. 

After performing this sequence of operations for each edge in $H_2$, we will have added $d\cdot\frac{U_{m-2}(x)}{U_{m-1}(x)}$ to each diagonal entry corresponding to a vertex in $V$. Also, each entry corresponding to an edge in $H_2$ will now be $\frac{1}{U_{m-1}(x)}$. This means we can order the vertices of $G$ so that $A_G-2xI$ has the same determinant as the block matrix 
$$\begin{bmatrix}
D      &       &       &             \\
       &R_m    &       &             \\
       &       &\ddots &             \\
       &       &       &R_m          \\
\end{bmatrix},$$ 
where $$D=A_{H_1}-\left(2x-\frac{dU_{m-2}(x)}{U_{m-1}(x)}\right)I+\frac{1}{U_{m-1}(x)}A_{H_2}$$ and there are $e(H_2)$ copies of the $(m-1)$ by $(m-1)$ matrix $$R_m=\begin{bmatrix}
-2x    &1      &       &       \\
1      &\ddots &\ddots &       \\
       &\ddots &-2x    &1      \\
       &       &1      &-2x    \\
\end{bmatrix}.$$ By induction on $m$ using \eqref{eq:U-recur}, $\det(R_m)=(-1)^{m-1}U_{m-1}(x)$, which gives the desired result.

\end{proof}

Apply \cref{lem:pathlen} to the graph $G$ constructed in \cref{thm:bounded-construction}.
Let 
\[
H_1 = \Cay(\Gamma, S)
\qquad \text{and}  \qquad
H_2 = \Cay(\Gamma, \{t,t^{-1}\}).
\]
Let $A$ and $B$ be the adjacency matrices of $H_1$ and $H_2$ respectively. 
Recall $x U_{m-1}(x) - U_{m-2}(x) = T_m(x)$ from \eqref{eq:T-U}. Since $e(H_2)=\abs{\Gamma}$ and $d=2$ in \cref{lem:pathlen}, we get
\begin{equation}\label{eq:sep-det}
\det(A_G - 2x I) = \pm\det\paren{ U_{m-1}(x) A - 2 T_m (x) I + B}.
\end{equation}

Applying the decomposition of the regular representation into irreducible representations, we find that 
\begin{equation}\label{eq:bounded-decomp}
U_{m-1}(x)A - 2T_m(x)I+ B
= 
\bigoplus_{\rho\text{ irred rep of $\Gamma$}}M_\rho(x)^{\oplus\dim\rho}
\end{equation}
with
\begin{equation}\label{eq:M}
M_\rho (x) = U_{m-1}(x)A_\rho - 2T_m(x) I + B_\rho,
\end{equation}
where, as earlier, we write
$
A_\rho = \sum_{h\in S}\rho(h)$ and $B_\rho = \rho(t) + \rho(t^{-1})$.

As in the earlier section, by \cref{lem:double-coset-irrep}, $\Ind_\Pi^\Gamma(\triv)$ is the direct sum of the trivial 1-dimensional representation and an irreducible representation of dimension $\abs{\Gamma}/\abs{\Pi}-1$, which we call $\varphi$.
For both these representations $\rho\in\{\triv,\varphi\}$, one has $\norm{A_\rho} = \abs{S}$.
For any other irreducible representation $\rho$ of $\Gamma$, one has $\norm{A_\rho} \le \abs{S} - \kappa$ due to the spectral gap of $A$.

Let $y_0$ be the largest real number satsifying
\[
2T_m(y_0)+2-\abs{S}U_{m-1}(y_0) = 0.
\]
Since $2T_m(1)+2-\abs{S}U_{m-1}(1)=4-m\abs{S}<0$ but $2T_m(y)+2-\abs{S}U_{m-1}(y)$ is positive as $y$ tends to infinity, $y_0$ exists and is greater than $1$.

\begin{lemma}\label{lem:irrep-gap} 
Let $\rho$ be an irreducible representation of $\Gamma$, and $m\geq 4$ be an integer. 
\begin{enumerate}
	\item[(a)] If $\rho$ is not contained in $\Ind_\Pi^\Gamma(\triv)$ and $\det M_\rho(x) = 0$, then
	\[
		2T_m(x) \le U_{m-1}(x) (\abs{S}-\kappa) + 2.
	\]
	\item[(b)] 
	If $\rho$ is contained in $\Ind_\Pi^\Gamma(\triv)$, then there exists some $y \ge y_0$ with $\det M_\rho(y) = 0$.
\end{enumerate}
\end{lemma}

\begin{proof}
(a) By the above, for $\rho$ not contained in $\Ind_\Pi^\Gamma(\triv)$, $\|A_\rho\|\leq \abs{S}-\kappa$. If $\det M_\rho(x) = 0$, then $2T_m (x)$ is an eigenvalue of $U_{m-1}(x) A_\rho + B_\rho$ (see \cref{eq:M}). Since $\norm{B_\rho} \le 2$, the conclusion follows by Weyl's inequality.
    
(b) For $\rho$ contained in $\Ind_\Pi^\Gamma(\triv)$, $A_\rho$ has $\abs{S}$ as an eigenvalue; let $\bm v$ be the associated normalized eigenvector. Then, consider
    \[
    \bm v^TM_\rho(y)\bm v=\abs{S}U_{m-1}(y)-2T_m(y)+\bm v^TB_\rho\bm v
    \]
    as a continuous function of $y$. As $T_m$ and $U_{m-1}$ are polynomials of degrees $m$ and $m-1$, respectively, the right-hand side is eventually negative as $y$ tends to infinity. On the other hand, at $y=y_0$, we have that
    \[
    \bm v^TM_\rho(y_0)\bm v=\abs{S}U_{m-1}(y_0)-2T_m(y_0)+\bm v^TB_\rho\bm v=2+\bm v^TB_\rho\bm v\geq 0,
    \]
    where we have used that all eigenvalues of $B_\rho$ are in $[-2,2]$. As a result, there exists some $y\geq y_0$ for which $\bm v^TM_\rho(y)\bm v=0$ and thus $\det M_\rho(y)=0$, as desired.
\end{proof}

\begin{lemma} \label{lem:cheby-gap}
	If there exists some $x>y_0$ with
	\begin{equation} \label{eq:cheby-gap-x}
		2T_m(x) \le U_{m-1}(x) (\abs{S}-\kappa) + 2,
	\end{equation}
	then $\abs{S}^{m-1} < 4/\kappa$.
\end{lemma}

\begin{proof} Recall that $y_0>1$. By \cref{lem:cheb-facts}(a), we have
$$\abs{S}-\frac4{U_{m-1}(y_0)}=\frac{2T_m(y_0)-2}{U_{m-1}(y_0)}<\frac{2T_m(x)-2}{U_{m-1}(x)}\leq \abs{S}-\kappa.$$
So $U_{m-1}(y_0)<\frac 4\kappa$.
Now, simplifying and applying \cref{lem:cheb-facts}(b) at $z=y_0$ gives 
$$\abs{S}^{m-1}=\left(\frac{2T_m(y_0)+2}{U_{m-1}(y_0)}\right)^{m-1}\leq U_{m-1}(y_0)<\frac 4\kappa,$$
the desired bound.
\end{proof}

\begin{proof}[Proof of \cref{thm:bounded-construction}]
Via \cref{eq:sep-det,eq:bounded-decomp}, the eigenvalues (counting multiplicities) of $G$ are $2x$ where $x$ is a root of $M_\rho(x)$, each repeated $\dim \rho$ times, ranging over all irreducible representations $\rho$ of $\Gamma$.
The trivial 1-dimensional representation of $\Gamma$ gives the the largest eigenvalue of $G$ with multiplicity one. 
By \cref{lem:double-coset-irrep}, there is one other irreducible representation $\varphi$ of $\Ind_\Pi^\Gamma(\triv)$. 
By \cref{lem:irrep-gap}, there is some $y \ge y_0$ satisfying $\det M_\varphi(y) = 0$.
We take the largest such $y$, and claim that $2y$ is the second-largest eigenvalue of $A_G$.
By definition, no other root of $\det M_\varphi$ exceeds $y$. The roots of $\det M_\triv$ are exactly the roots of $2T_m(x)-2-|S|U_{m-1}(x)$. By \cref{lem:cheb-facts}(a), only one such root (which gives the top eigenvalue of $G$) is greater than $1$. So only one of the roots of $\det M_\triv$ exceeds $y$.
Now suppose $x$ is such that $\det M_\rho(x) = 0$ for some irreducible representation $\rho$ other than $\varphi$ and $\triv$ (so that $\norm{A_\rho} \le \abs{S}-\kappa$).
Then $x$ satisfies \cref{eq:cheby-gap-x} by \cref{lem:irrep-gap}(a).
It follows by \cref{lem:cheby-gap}, due to the hypothesis $\abs{S}^{m-1} \ge 4/\kappa$, that $x \le y_0\le y$. 
Therefore, $2y$ is indeed the second largest eigenvalue of $G$, and it has multiplicity at least $\varphi = \abs{\Gamma}/\abs{\Pi} - 1$ times in $G$.
\end{proof}

\begin{proof}[Proof of Theorem \ref{thm:bounded-intro}] 
Let $\Pi = \FF_q^\times$ for some prime power $q\geq 4$.
Suppose $S=\{s,s^{-1}\}$ generates the cyclic group $\Pi$.
Using the action of $\Pi = \FF_q^\times$ on $\FF_q$ by multiplication, set $\Gamma = \FF_q^\times \ltimes \FF_q$ ($\Gamma$ is also the affine group of $\FF_q$).
Then $\abs{\Pi \backslash \Gamma /\Pi} = 2$ since this is the number of $\Pi$-orbits in $\FF_q$. 
As $\Cay(\Gamma,S)$ is simply $q$ copies of a $(q-1)$-cycle,
$$\kappa=2-2\cos\left(\frac{2\pi}{q-1}\right)\geq \frac{32}{(q-1)^2}.$$
Applying \cref{thm:bounded-construction} with $m = \lceil 2\log_2(q-1)\rceil-2$ (so that the hypothesis $\abs{S}^{m-1}\ge 4/\kappa$ is satisfied) gives a graph with maximum degree $4$ on $q(q-1)m$ vertices with second eigenvalue multiplicity at least $q-1$, as claimed.
\end{proof}

\section{Approximate second eigenvalue multiplicity}
\label{sec:approx1}
In this section we show \cref{thm:approx-intro}, our result on approximate second eigenvalue multiplicity. Our construction is as follows. We sample a random $3$-regular graph $H$ on $N$ vertices. We replace each vertex of $H$ by a copy of $K_4$ with vertices labeled by $\{1,2,3,4\}$, and replace each edge of $H$ with four disjoint paths of length $\ell = \Theta(\log\log N)$ connecting identically labeled vertices of $K_4$. 
The resulting graph $G = G(H, \ell)$ will have the desired properties with high probability. 

\medskip

Recall the following result about the eigenvalues of random regular graphs.

\begin{theorem}[Friedman's second eigenvalue theorem {\cite[Theorem 1.4]{Fri08}}]
\label{thm:Friedman}
Let $H$ be a $3$-regular graph on $n$ labeled vertices chosen uniformly at random (with $n$ even).
For every fixed $\epsilon > 0$, with probability $1 - o(1)$ as $n \to \infty$, we have
$$\lambda_2(H) \leq 2\sqrt{2} + \epsilon.$$
\end{theorem}

\begin{theorem}[Kesten--McKay law {\cite{Kes59,M81}}]
\label{thm:K-M}
Let $H_n$ be a $3$-regular graph on $n$ labeled vertices chosen uniformly at random.
Then, with probability $1$, as $n \to \infty$ along even integers, the empirical distribution of the eigenvalues of $H_n$ converges to the probability distribution with density function
\[
z \mapsto  \frac{3\sqrt{8 - z^2}}{2\pi(9 - z^2)} 1_{[-2\sqrt{2},2\sqrt{2}]}(z).
\]
\end{theorem}

We need graphs with the following properties.

\begin{corollary}
\label{cor:basic-expander}
For any $\epsilon > 0$, there exists some $a(\epsilon) > 0$ such that for all sufficiently large even integer $n$, there exists a connected $3$-regular graph $H$ on $n$ vertices such that

(1) $H$ is connected and has spectral gap at least $0.01$, and 

(2) $H$ has at least $a(\epsilon)n$ eigenvalues in the interval $[(1 - \epsilon)\lambda_2(H), \lambda_2(H)]$.
\end{corollary}
\begin{proof}
Let $H$ be a random $3$-regular graph on $n$ vertices, chosen uniformly from all labeled $3$-regular graphs on $n$ vertices. By \cref{thm:Friedman} and \cref{thm:K-M}, if we let 
$$a(\epsilon) = \frac{1}{2}\int_{2\sqrt{2} - \sqrt{2}\epsilon }^{2\sqrt{2}} \frac{3 \sqrt{8 - z^2}}{2\pi(9 - z^2)}dz$$
then with probability $1 - o(1)$,

(1) $H$ is connected (since the spectral gap is positive), 

(2) $\lambda_2(H) \leq 2\sqrt{2} + \epsilon / 2$, and

(3) $m_H([(1 - \epsilon / 2)\cdot 2\sqrt{2}, 2\sqrt{2}]) > a(\epsilon)n.$

So $H$ with desired properties exists for all sufficiently large even $n$.
\end{proof}

\begin{construction}
\label{con:second}
Let $H$ be a $3$-regular graph on $n$ vertices. We construct a new graph $G$ from $H$. Let the vertex set $V$ be $\{(i, v): v \in V(H), i \in \{1,2,3,4\}\}$. Let the graph $H_1$ be the disjoint union of $n$ copies of $K_4$ with vertices $\{(i, v):i \in \{1,2,3,4\}\}$ for each $v \in V(H)$, and let the graph $H_2$ be the disjoint union of $4$ copies of $H$ with vertices $\{(i, v): v\in V(H)\}$ for each $i \in [4]$. We construct $G = G(H, \ell)$ as graph obtained by first overlaying $H_1$ and $H_2$ and then replacing each edge of $H_2$ with a path of length $\ell$.
\end{construction}
In other words, we take the Cartesian product $H \square K_4$, then replace each edge in the image derived from an edge of $H$ with a path of length $\ell$.

Our main claim is the following eigenvalue distribution result for $G$.
\begin{theorem}
\label{thm:second}
For any large even integer $n$, positive integer $\ell > 10$ and positive constant $\epsilon > 0$, let $H$ be a $3$-regular graph on $n$ vertices satisfying (1) and (2) of  \cref{cor:basic-expander}. let $G = G(H, \ell)$ be as in \cref{con:second}. Then, there exist absolute constants $C_1 > 0$ and $\alpha_0 > 1$ such that

(1) $G$ is connected, and the maximum degree of $G$ is $6$.

(2) The spectral gap $\kappa$ of $G$ is at least $C_1\alpha_0^{-\ell}$.

(3) $G$ has at least $a(\epsilon)n$ eigenvalues in the interval
$$\left[\lambda_2(G)(1 - C_1^{-1} \epsilon\alpha_0^{-\ell}), \lambda_2(G)\right],$$
where $a(\epsilon)$ is defined in \cref{cor:basic-expander}.
\end{theorem}
We first see that \cref{thm:second} implies \cref{thm:approx-intro}.
\begin{proof}[Proof of \cref{thm:approx-intro}]
Let $\epsilon = \delta C_1^2 / 2\alpha_0 C$, and let $N$ be any large even integer. Take $\ell = \floor{\log_{\alpha_0} (\frac{C_1}C\log N)}$. Let $H$ be a $3$-regular graph on $N$ vertices satisfying (1) and (2) of  \cref{cor:basic-expander}. We consider $G = G(H, \ell)$. Then $G$ is a connected graph with $n = (6\ell - 2)N$ vertices, and it is straightforward to verify that $G$ satisfies the desired properties.
\end{proof}
Now, we prove \cref{thm:second}. We first give a characterization of the relevant portion of the spectrum of $G = G(H, \ell)$.
\begin{lemma}
\label{lem:f}
Let $H$ be any $3$-regular graph, $\ell > 10$ be any positive integer, and $G = G(H, \ell)$. Then $\lambda > 2$ is an eigenvalue of $A_G$ if and only if
$$f(\lambda) = \left(\lambda - \frac{3U_{\ell - 2}(\lambda / 2) }{U_{\ell - 1}(\lambda / 2)} - 3\right)U_{\ell - 1}(\lambda / 2)$$
is an eigenvalue of $A_H$. Furthermore, the multiplicity of $\lambda$ in $A_G$ is at least the multiplicity of $f(\lambda)$ in $A_H$.
\end{lemma}
\begin{proof}
We use \cref{lem:pathlen} on $H_1$ and $H_2$, where we recall $H_1$ is $k$ disjoint copies of $K_4$ and $H_2$ is $4$ disjoint copies of $H$, to obtain
$$
\det(A_G- 2x I)=\pm U_{\ell-1}(x)^{e(H_2)}\det\left( A_{H_1}-\left(2x-\frac{3U_{\ell-2}(x)}{U_{\ell - 1}(x)}\right)I+\frac{1}{U_{\ell-1}(x)}A_{H_2}\right).$$
The multiplicity of $\lambda$ in $A_G$ is equal to the multiplicity of $x = \lambda / 2$ as a root of the polynomial $\det(A_G- 2x I)$, which is equal to the multiplicity of $x = \lambda / 2$ as a root of the polynomial
$$
\pm U_{\ell-1}(x)^{e(H_2)}\det\left(A_{H_1}-\left(2x-\frac{3U_{\ell-2}(x)}{U_{\ell - 1}(x)}\right)I+\frac{1}{U_{\ell-1}(x)}A_{H_2}\right).$$
Since $U_{\ell-1}$ has no roots outside of $[-1,1]$, the multiplicity of $\lambda$ in $A_G$ is equal to the multiplicity of $x = \lambda / 2$ as a root of the rational function
\begin{equation}
\label{eq:determinant}
\det\left(A_{H_1}-\left(2x-\frac{3U_{\ell-2}(x)}{U_{\ell - 1}(x)}\right)I+\frac{1}{U_{\ell-1}(x)}A_{H_2}\right).
\end{equation}
We regard the adjacency matrices of $H_1$ and $H_2$ as acting on the vector space $\mathbf{R}^{V}$, where $V$ is the common vertex set of $H_1$ and $H_2$. Note that
$$W = \{\bm v\in \mathbf{R}^{V}: \bm v(i, u) = \bm v(j, u), \forall i,j\in [4], u \in V(H)\}$$
is an invariant subspace of $A_{H_1}$ and $A_{H_2}$, and so is $W^{\perp}$.
Therefore
\begin{align*}
    &\det\left(A_{H_1}-\left(2x-\frac{3U_{\ell-2}(x)}{U_{\ell - 1}(x)}\right)I+\frac{1}{U_{\ell-1}(x)}A_{H_2}\right) \\ =& 
    \det\left(A_{H_1}|_W-\left(2x-\frac{3U_{\ell-2}(x)}{U_{\ell - 1}(x)}\right)I+\frac{1}{U_{\ell-1}(x)}A_{H_2}|_W\right)\cdot\\&\det\left(A_{H_1}|_{W^{\perp}}-\left(2x-\frac{3U_{\ell-2}(x)}{U_{\ell - 1}(x)}\right)I+\frac{1}{U_{\ell-1}(x)}A_{H_2}|_{W^{\perp}}\right).
\end{align*}
Furthermore, for any $\bm u$ in $W^{\perp}$, we have
$$\bm u^T A_{H_1} \bm u = - \bm u^T \bm u,\qquad \bm u^T A_{H_2} \bm u \leq 3 \bm u^T \bm u,$$
where the second property holds because $H_2$ is $3$-regular. Thus, for all $\lambda > 2$ and $\bm u\in W^\perp$, we have
$$\bm u^T\left(A_{H_1} - \left(\lambda - \frac{3U_{\ell - 2}(\lambda / 2)}{U_{\ell - 1}(\lambda / 2)}\right)I + \frac{1}{U_{\ell - 1}(\lambda / 2)} A_{H_2}\right) \bm u \leq \left( -1 - \left(\lambda - \frac{3U_{\ell - 2}(\lambda / 2)}{U_{\ell - 1}(\lambda / 2)}\right) +  \frac{3}{U_{\ell - 1}(\lambda / 2)}\right)\bm u^T \bm u.$$
Since $\lambda > 2$, we can show by induction on $\ell$ that 
$$U_{\ell - 2}(\lambda / 2) + 1 < U_{\ell-1}(\lambda / 2).$$
Indeed, the base case $\ell = 2$ is clear, and for the induction step we have
$$U_{\ell-1}(\lambda / 2) = \lambda U_{\ell - 2}(\lambda / 2) - U_{\ell - 3}(\lambda / 2)$$
which gives
$$U_{\ell-1}(\lambda / 2) - U_{\ell - 2}(\lambda / 2) > U_{\ell - 2}(\lambda / 2) - U_{\ell - 3}(\lambda / 2).$$
So, we have
$$-1 - \left(\lambda - \frac{3U_{\ell - 2}(\lambda / 2)}{U_{\ell - 1}(\lambda / 2)}\right) +  \frac{3}{U_{\ell - 1}(\lambda / 2)} = \frac{3(U_{\ell - 2}(\lambda / 2) + 1)}{U_{\ell-1}(\lambda / 2)} - \lambda - 1 < 0,$$
and thus for any $\bm u$ in $W^{\perp}$,
$$\bm u^T\left(A_{H_1} - \left(\lambda - \frac{3U_{\ell - 2}(\lambda / 2)}{U_{\ell - 1}(\lambda / 2)}\right)I + \frac{1}{U_{\ell - 1}(\lambda / 2)} A_{H_2}\right) \bm u < 0 .$$
Therefore, for $x = \lambda / 2$ we have
$$\det\left(A_{H_1}|_{W^{\perp}}-\left(2x-\frac{3U_{\ell-2}(x)}{U_{\ell - 1}(x)}\right)I+\frac{1}{U_{\ell-1}(x)}A_{H_2}|_{W^{\perp}}\right) \neq 0.$$
So the multiplicity of $\lambda$ in $A_G$ is equal to the multiplicity of the root $x = \lambda / 2$ in the rational function
\begin{equation*}
\det\left(A_{H_1}|_W-\left(2x-\frac{3U_{\ell-2}(x)}{U_{\ell - 1}(x)}\right)I+\frac{1}{U_{\ell-1}(x)}A_{H_2}|_W\right).
\end{equation*}
Simplifying, we obtain
\begin{align*}
   \label{eq:determinant2}
\MoveEqLeft
A_{H_1}|_W-\left(2x-\frac{3U_{\ell-2}(x)}{U_{\ell - 1}(x)}\right)I+\frac{1}{U_{\ell-1}(x)}A_{H_2}|_W
\\
&= 3I -\left(2x-\frac{3U_{\ell-2}(x)}{U_{\ell - 1}(x)}\right)I+\frac{1}{U_{\ell-1}(x)}A_H \\
&= \frac{1}{U_{\ell - 1}(x)}(A_H - f(2x) I).  
\end{align*}
Thus the multiplicity of $\lambda$ in $A_G$ is equal to the multiplicity of $x = \lambda / 2$ as a root of $\det(A_H - f(2x)I)$, which is equal to the multiplicity of $f(\lambda)$ in $A_H$ multiplied by the multiplicity of $x = \lambda / 2$ as a root of $f(2x) - f(\lambda)$. In particular, $\lambda$ is an eigenvalue of $A_G$ if and only if $f(\lambda)$ is an eigenvalue of $A_H$, and the multiplicity of $\lambda$ in $A_G$ is at least the multiplicity of $f(\lambda)$ in $A_H$. 
\end{proof}

\begin{proof}[Proof of \cref{thm:second}]
Condition (1) follows directly from the construction of $G = G(H, \ell)$. We now prove (2) and (3). We take $C_1 = 0.001$ and $\alpha_0 = \frac{3 + \sqrt{17}}{2}$. By \cref{lem:f}, $\lambda > 2$ is an eigenvalue of $A_G$ if and only if $f(\lambda)$ is an eigenvalue of $A_H$, and the multiplicity of $f(\lambda)$ is at least that of $\lambda$. We now give the following lemmas on the behavior of $f$. These lemmas roughly assert that, in the range of interest, $f$ behaves sufficiently like a linear function. Both lemmas assume $\ell > 10$.
\begin{restatable}{lemma}{fbashfirst}
\label{lem:f-bash-1}
 We have $f(3) < 0$, $\lim_{\lambda\to\infty} f(\lambda) = \infty$ and if $\lambda > 3$ satisfies $f(\lambda) \geq 0$, then $f'(\lambda) > 0$.

Thus, there exists a unique $\lambda_* > 3$ with $f(\lambda_*) = 0$, and $f$ is increasing on $[\lambda_*, \infty)$.
\end{restatable}
\begin{restatable}{lemma}{fbashsecond}
\label{lem:f-bash-2}
For any $\lambda \in [\lambda_*, \infty)$, if we let $\alpha > 2$ satisfy $\lambda = \alpha + \alpha^{-1}$, then
\begin{enumerate}
    \item if $\alpha > \alpha_0 = \frac{3 + \sqrt{17}}{2}$, then we have
$$f(\lambda) \geq (\alpha - \alpha_0)\alpha_0^{\ell - 1}.$$
    \item If $f(\lambda) < 5$, then we have 
    $$0.01\alpha_0^{\ell} < f'(\lambda) < 3\alpha_0^{\ell}.$$
\end{enumerate}
\end{restatable}

We defer the proofs of these lemmas to \cref{sec:approx2}.

By \cref{lem:f-bash-1}, $\lambda_1$ and $\lambda_2$, the first and second eigenvalue of $A_G$, are both greater than $\lambda_*$, and $f(\lambda_1)$ and $f(\lambda_2)$ are the first and the second eigenvalue of $A_H$ respectively. By the condition imposed on the spectral gap of $H$, we get
$$f(\lambda_1) - f(\lambda_2) \geq 0.01.$$
Also note that $f(\lambda_1)$ is equal to the top eigenvalue $3$ of $A_H$. By \cref{lem:f-bash-2}, for any $\lambda \in [\lambda_2, \lambda_1]$ we have $f'(\lambda) < 3 \alpha_0^{\ell}$. This lets us conclude (2) of \cref{thm:second} as 
$$\kappa = \lambda_1 - \lambda_2 \geq \frac{f(\lambda_1) - f(\lambda_2)}{3\alpha_0^{\ell}} \geq 0.001\alpha_0^{-\ell}.$$

We next demonstrate the eigenvalue distribution result. Let $\mu_2, \mu_3, \dots, \mu_{k}$ be the eigenvalues of $H$ (with multiplicity) in the interval $[(1 - \epsilon)\mu_2, \mu_2]$. By our assumption on $H$, we have $k-1 \geq a(\epsilon)n$. As $f(\lambda_*) = 0$ and $f$ is increasing on $[\lambda_*, \infty)$, for each $2 \leq i \leq k$ there exists some $\lambda_i \in [\lambda_*, \infty)$ such that $f(\lambda_i) = \mu_i$. By \cref{lem:f}, the eigenvalues of $G$ contain $\lambda_i$ (including multiplicity). Furthermore we note
$$f(\lambda_2) - f(\lambda_i) < \epsilon\mu_2.$$
By \cref{lem:f-bash-2}, for any $\lambda \in [\lambda_*, \lambda_1]$, we have $f'(\lambda) > 0.01 \alpha_0^{\ell}$.
Therefore we obtain
$$f(\lambda_2) - f(\lambda_i) > 0.01 \alpha_0^{\ell}(\lambda_2 - \lambda_i).$$
So we conclude that
$$\lambda_2 - \lambda_i < 100\epsilon \mu_2 \alpha_0^{-\ell}<300\epsilon \alpha_0^{-\ell}\lambda_2,$$
where we have used $\mu_2<3$ and $\lambda_2>2$. Therefore the $\lambda_i$'s are $k - 1$ distinct eigenvalues of $G$ in the interval
$$\left[(1 - 300\epsilon \alpha_0^{-\ell})\lambda_2, \lambda_2\right]$$
which shows (3) of \cref{thm:second}.
\end{proof}

\section*{Acknowledgments}
We thank Alex Lubotzky and Naser Talebizadeh Sardari for some clarifications regarding the constructions of Ramanujan graphs.


\appendix
\section{Proof of \cref{thm:JYTZZ-approx}}
\label{app:proof}
In this appendix, we explain how to modify the argument in \cite{JTYZZ1} to show \cref{thm:JYTZZ-approx}.

Let $G$ be any connected graph with $n$ vertices, maximum degree $\Delta$ and second eigenvalue $\lambda_2 = \lambda_2(G)$. 
Recall the argument of \cite[Section 4]{JTYZZ1}: choose some constant $c > 0$ depending only $\Delta$, and define parameters $r_1 = \floor{c\log\log n}$ and $r_2 = \floor{c\log n}$. First, remove a set $V_0\cup U$ of at most $O_{\Delta}(n / \log\log n)$ vertices, including an $r_1$-net $V_0$ plus a set $U$ of at most $n^{1 / 2}$ other vertices. Then, \cite[Section 4]{JTYZZ1} shows that the resulting graph $H$ satisfies
$$\sum_{i = 1}^{\abs{H}} \lambda_i(H)^{2r_2} \leq (\lambda_2^{2r_1} - 1)^{r_2 / r_1} n.$$
In \cite{JTYZZ1}, the above is used to directly conclude that
$$m_H(\lambda_2)  \leq \lambda_2^{-2r_2}\sum_{i = 1}^{\abs{H}} \lambda_i(H)^{2r_2} \leq (1 - \lambda_2^{-2r_1})^{r_2 / r_1} n \leq e^{-\sqrt{\log n}}n.$$
Here, we take a slightly broader view of this equation. For any constant $K > 0$, we observe that
$$m_H \sqb{\paren{1 - \frac{K}{\log n}} \lambda_2, \lambda_2} \leq \paren{\paren{1 - \frac{K}{\log n}}\lambda_2}^{-2r_2}\sum_{i = 1}^{\abs{H}} \lambda_i(H)^{2r_2}.$$
Note that, for all sufficiently large $n$, we have
$$\paren{1 - \frac{K}{\log n}}^{-2r_2} \leq \paren{1 + 2\frac{K}{\log n}}^{2r_2} \leq \exp\paren{2\frac{K}{\log n} \cdot 2r_2} \leq \exp(4Kc).$$
Therefore we conclude that
$$m_H \sqb{\paren{1 - \frac{K}{\log n}} \lambda_2, \lambda_2}  \leq e^{4Kc}(1 - \lambda_2^{-2r_1})^{r_2 / r_1} n \leq e^{4Kc} e^{-\sqrt{\log n}}n = O_{\Delta, K}\paren{\frac{n}{\log \log n}}.$$
Finally, by the Cauchy eigenvalue interlacing theorem, since $H$ is obtained from $G$ by removing a set $V_0\cup U$ of at most $O_{\Delta}(n / \log\log n)$ vertices, we have
$$m_G\sqb{\paren{1 - \frac{K}{\log n}} \lambda_2, \lambda_2}\leq m_H\sqb{\paren{1 - \frac{K}{\log n}} \lambda_2, \lambda_2} + \abs{V_0\cup U} = O_{\Delta, K}\paren{\frac{n}{\log \log n}}.$$
This shows \cref{thm:JYTZZ-approx}.

\section{Proof of \cref{lem:f-bash-1} and \cref{lem:f-bash-2}}
\label{sec:approx2}
Let $\ell > 10$ be an integer. Recall that in \cref{sec:approx1}, we defined the function
$$f(\lambda) = \left(\lambda - \frac{3U_{\ell - 2}(\lambda / 2) }{U_{\ell - 1}(\lambda / 2)} - 3\right)U_{\ell - 1}(\lambda / 2)$$
and made the following claims of its behavior.
\fbashfirst*
\fbashsecond*
Here we prove these estimates.
\begin{proof}[Proof of \cref{lem:f-bash-1}]
First,
$$f(3) = -\frac{3U_{\ell - 2}(3 / 2) }{U_{\ell - 1}(3/ 2)}U_{\ell - 1}(3 / 2) < 0.$$
Now, we show that both factors in the definition of $f$ are increasing in $\lambda$ for $\lambda>2$. For the first, it suffices to show that
$$\frac{U_{\ell-2}(\lambda/2)}{U_{\ell-1}(\lambda/2)}$$
is decreasing in $\lambda$, which holds by induction on $\ell$ with a base case of $\ell=2$ and the recurrence
$$\frac{U_m(\lambda/2)}{U_{m+1}(\lambda/2)}=\frac{1}{\lambda-\frac{U_{m-1}(\lambda/2)}{U_m(\lambda/2)}}$$
derived from \cref{eq:U-recur} (along with the observation that $U_m$ has no zeros outside $[-1,1]$ for any $m$). For the second, we have writing $\lambda=\alpha+1/\alpha$ that
$$U_{\ell-1}(\lambda / 2)=\frac{\alpha^\ell-\alpha^{-\ell}}{\alpha-\alpha^{-1}}=\sum_{\substack{1-\ell\leq i\leq \ell-1\\i\equiv \ell-1\pmod 2}}\alpha^i.$$
Pairing the terms corresponding to $i$ and $-i$ shows that the second factor is increasing in $\alpha$ and thus $\lambda$. Letting $\lambda^*>3$ be the sole root of the first factor of $f$ in $[3,\infty)$ finishes the proof.
\end{proof}
\begin{proof}[Proof of \cref{lem:f-bash-2}]
The first claim is due to the relation
\begin{align*}
f(\lambda) &= \left(\lambda - 3 - \frac{3(\alpha^{\ell - 1} - \alpha^{-(\ell - 1)})}{\alpha^{\ell} - \alpha^{-\ell}}\right)\frac{\alpha^{\ell} - \alpha^{-\ell}}{\alpha - \alpha^{-1}} \\
&= \left(\alpha + \alpha^{-1} - 3 - \frac{3(\alpha^{\ell - 1} - \alpha^{-(\ell - 1)})}{\alpha^{\ell} - \alpha^{-\ell}}\right)\frac{\alpha^{\ell} - \alpha^{-\ell}}{\alpha - \alpha^{-1}} \\
&= \left(\alpha - 2\alpha^{-1} - 3 - 3\left(\frac{(\alpha^{\ell - 1} - \alpha^{-(\ell - 1)})}{\alpha^{\ell} - \alpha^{-\ell}} - \alpha^{-1}\right)\right)\frac{\alpha^{\ell} - \alpha^{-\ell}}{\alpha - \alpha^{-1}} \\
&= \left(\alpha - 2\alpha^{-1} - 3 + \frac{3(\alpha^{-\ell+1}-\alpha^{-\ell-1})}{\alpha^{\ell} - \alpha^{-\ell}}\right)\frac{\alpha^{\ell} - \alpha^{-\ell}}{\alpha - \alpha^{-1}}.
\end{align*}
We note that $\alpha_0 = 2\alpha_0^{-1} + 3$ and
$$\frac{\alpha^{\ell} - \alpha^{-\ell}}{\alpha - \alpha^{-1}} \geq \alpha^{\ell - 1} \geq \alpha_0^{\ell - 1}.$$
Therefore
$$f(\lambda) \geq (\alpha - 2\alpha^{-1} - 3) \alpha_0^{\ell - 1} \geq (\alpha - 2\alpha_0^{-1} - 3) \alpha_0^{\ell - 1} = (\alpha - \alpha_0)\alpha_0^{\ell - 1}.$$
The second claim is proven by the product rule for derivatives. First we show the upper bound. As $f(\lambda) < 5$, the first claim gives
$$\alpha - \alpha_0 < 5 \alpha_0^{-\ell + 1}.$$
Now note that by the product rule, we have
\begin{align*}
 f'(\lambda) = \frac{d \alpha}{d\lambda} \left(A'(\alpha) B(\alpha) + A(\alpha)B'(\alpha)\right)
\end{align*}
where
$$A(\alpha) = \alpha - 2\alpha^{-1} - 3 + \frac{3(\alpha^{-\ell+1}-\alpha^{-\ell-1})}{\alpha^{\ell} - \alpha^{-\ell}},\qquad B(\alpha) = \frac{\alpha^{\ell} - \alpha^{-\ell}}{\alpha - \alpha^{-1}}.$$
We first note that if $\alpha \geq \alpha_0$, then
$$\alpha - 2\alpha^{-1} - 3 = \alpha - 2\alpha^{-1} - (\alpha_0 - 2\alpha_0^{-1}) = (\alpha - \alpha_0)\left(1 + \frac{2}{\alpha \alpha_0}\right) \leq (\alpha - \alpha_0)\left(1 + \frac{2}{ \alpha_0^2}\right) < 2(\alpha - \alpha_0).$$
while if $\alpha < \alpha_0$ then $\alpha - 2\alpha^{-1} - 3  \leq 0$. As $\alpha - \alpha_0 < 5 \alpha_0^{-\ell + 1}$, we have in both cases that
$$\alpha - 2\alpha^{-1} - 3 \leq 10 \alpha_0^{-\ell + 1}.$$
As $\ell > 10$, we can estimate $\alpha - \alpha_0 < 5 \alpha_0^{-\ell + 2} < 0.1\ell^{-1}$.
Therefore we have
$$\alpha^{\ell} < (\alpha_0 + 0.1\ell^{-1})^{\ell} <\alpha_0^{\ell}(1 + 0.05\ell^{-1})^{\ell} < e^{0.05}\alpha_0^{\ell} < 1.1\alpha_0^{\ell}.$$
This implies
$$\alpha - 2\alpha^{-1} - 3 \leq 10(\alpha_0^\ell)^{(-\ell+1)/\ell} \leq 10\cdot 1.1^{(-\ell+1)/\ell}\alpha^{-\ell+1}= 11 \alpha^{-\ell + 1}.$$
Furthermore we have
$$\frac{3(\alpha^{-\ell+1}-\alpha^{-\ell-1})}{\alpha^{\ell} - \alpha^{-\ell}} \leq \frac{3\alpha^{-\ell+1}}{3} = \alpha^{-\ell + 1}.$$
Therefore we conclude that
$$A(\alpha) = \alpha - 2\alpha^{-1} - 3 + 3\frac{\alpha^{-\ell+1}-\alpha^{-\ell-1}}{\alpha^{\ell} - \alpha^{-\ell}} < 12\alpha^{-\ell + 1}.$$
We then note that
$$B(\alpha) = \sum_{i = 0}^{\ell - 1} \alpha^{\ell - 1 - 2i}.$$
By directly taking the derivative we observe that
$$B'(\alpha) = \sum_{i = 0}^{\ell - 1} (\ell - 1 - 2i)\alpha^{\ell - 2 - 2i} < \ell \sum_{i = 0}^{\ell - 1} \alpha^{\ell - 2 - 2i} <\ell \frac{\alpha^\ell}{\alpha^2-1}< 2\ell \alpha^{\ell - 1}.$$
We multiply the above estimates to obtain
$$A(\alpha)B'(\alpha) < 24\ell.$$
On the other hand, we see that the last summand of $A$,
$$\frac{3(\alpha^{-\ell+1}-\alpha^{-\ell-1})}{\alpha^{\ell} - \alpha^{-\ell}} = 3\alpha^{-\ell}U_{\ell - 1}(\lambda)^{-1},$$
is decreasing in $\alpha$, so we have
$$A'(\alpha) < (\alpha - 2\alpha^{-1} - 3)' = 1 + 2\alpha^{-2} < 1.5$$
and the direct estimate, as $\alpha > 2$,
$$B(\alpha) < \frac{\alpha^{\ell}}{\alpha - \alpha^{-1}} < \frac{2}{3}\alpha^{\ell}.$$ 
Therefore we conclude that $B(\alpha) < 0.8\alpha_0^{\ell}$, and thus
$$A'(\alpha)B(\alpha) < 1.2\alpha_0^{\ell}.$$
Summing and using the assumption that $\ell > 10$, we have
$$A'(\alpha) B(\alpha) + A(\alpha)B'(\alpha) < 24\ell + 1.2\alpha_0^{\ell} < 1.5\alpha_0^{\ell}.$$
Finally, note that as $\alpha > 2$, we have $\frac{d \alpha}{d\lambda} = \frac{1}{1 - \alpha^{-2}} < 2$, so we conclude that $f'(\lambda) < 3\alpha_0^{\ell}$ as desired.

The proof of the lower bound is similar to the above, except that the inequalities must be reversed. As $f(\lambda) > 0$, we have 
$$\alpha - 2\alpha^{-1} - 3 + 3\frac{\alpha^{-\ell+1}-\alpha^{-\ell-1}}{\alpha^{\ell} - \alpha^{-\ell}} > 0.$$
As $\alpha > 2$, the above implies
$$\alpha - 2\alpha^{-1} - 3 + 6\cdot 2^{-2\ell + 1} > 0.$$
If $\alpha < \alpha_0$ we have
$$\alpha - 2\alpha^{-1} - 3 < \alpha - 2\alpha_0^{-1} - 3 = \alpha - \alpha_0,$$
which implies $\alpha - \alpha_0 > - 6\cdot 2^{-2\ell + 1}$. If $\alpha\ge\alpha_0$, the above inequality also holds.

We now use the same formulas as before. Recall
\begin{align*}
 f'(\lambda) = \frac{d \alpha}{d\lambda} \left(A'(\alpha) B(\alpha) + A(\alpha)B'(\alpha)\right).
\end{align*}
Since we have $A(\alpha) > 0$ and $B'(\alpha) > 0$, we can estimate
\begin{align*}
 f'(\lambda) > \frac{d \alpha}{d\lambda} A'(\alpha) B(\alpha).
\end{align*}
We now have
$$A'(\alpha) = 1 + 2\alpha^{-2} + 3\left(\frac{\alpha^{-\ell+1}-\alpha^{-\ell-1}}{\alpha^{\ell} - \alpha^{-\ell}}\right)'.$$
Recalling $\ell > 10$ and $\alpha > 2$, we can loosely estimate the last term:
\begin{align*}
 \left(\frac{\alpha^{-\ell+1}-\alpha^{-\ell-1}}{\alpha^{\ell} - \alpha^{-\ell}}\right)'=& \frac{\left((-\ell + 1)\alpha^{-\ell+1} - (-\ell - 1)\alpha^{-\ell+1}\right)(\alpha^{\ell} - \alpha^{-\ell}) - \ell (\alpha^{-\ell+1}-\alpha^{-\ell-1})(\alpha^{\ell} + \alpha^{-\ell})}{\alpha(\alpha^{\ell} - \alpha^{-\ell})^2}  \\
 \geq& \frac{(-\ell + 1)\alpha^{-\ell+1}(\alpha^{\ell} - \alpha^{-\ell}) - \ell \alpha^{-\ell+1}(\alpha^{\ell} + \alpha^{-\ell})}{(\alpha^{\ell} - \alpha^{-\ell})^2} \cdot \alpha^{-1} \\
 \geq& \frac{-4\ell\alpha}{(\alpha^{\ell} - \alpha^{-\ell})^2} \cdot \alpha^{-1}  > -2\alpha^{-2}.
\end{align*}
Thus we conclude that $A'(\alpha) > 1$. On the other hand, we have
$$B(\alpha) = \sum_{i = 0}^{\ell - 1} \alpha^{\ell - 1 - 2i} > \alpha^{\ell - 1}.$$
Recall that $\alpha > \alpha_0 - 6\cdot 2^{-2\ell + 1}$, which implies $\alpha^{\ell - 1} > 0.5\alpha_0^{\ell - 1}$. We conclude that
$$B(\alpha) > 0.5\alpha_0^{\ell - 1} > 0.1 \alpha_0^{\ell}.$$
Finally, we have $\frac{d \alpha}{d\lambda} = \frac{1}{1 - \alpha^{-2}} > 1$. Multiplying the estimates, we conclude that $f'(\lambda) > 0.1\alpha_0^{\ell}$ as desired.
\end{proof}

\end{document}